\pdfoutput=1
\RequirePackage{ifpdf}
\ifpdf 
\documentclass[pdftex]{sigma}
\else
\documentclass{sigma}
\fi

\usepackage{faktor}
\usepackage[all]{xy}

\newcommand{\overbar}[1]{\mkern 1.5mu\overline{\mkern-1.5mu#1\mkern-1.5mu}\mkern 1.5mu}

\newcommand{\Q}{\mathbb{Q}}
\newcommand{\Z}{\mathbb{Z}}

\newcommand{\F}{\mathbb{F}}

\newcommand{\Fq}{\mathbb{F}_q}
\newcommand{\Fqbar}{\overbar{\mathbb{F}_q}}

\newcommand{\Aut}{\operatorname{Aut}}

\newcommand{\absGal}[1][K]{\operatorname{G}_{\overbar{#1}/#1}}
\newcommand{\id}{\operatorname{id}}

\newcommand{\Twist}{\operatorname{Twist}}

\newcommand{\charact}{\operatorname{char}}
\newcommand{\Ker}{\operatorname{Ker}}
\newcommand{\Etw}{E^{\operatorname{tw}}}
\newcommand{\Frob}{\operatorname{Fr}}
\newcommand{\lmr}{\langle -1 \rangle}
\newcommand{\comment}[1]{}
\newcommand{\map}[5]{
		\begin{aligned}
			&&&&#1\colon && #2 &&\longrightarrow	&&& #3 &&&&\\
			&&&& && #4 	&&\longmapsto	&&& #5 &&&&
		\end{aligned}
}

\newcommand{\Ma}[4]{\begin{Pmatrix}#1& #2\\#3& #4\end{pmatrix}}

\newtheorem{Theorem}{Theorem}[section]

\newtheorem{Proposition}[Theorem]{Proposition}

{ \theoremstyle{definition}

\newtheorem{Example}[Theorem]{Example}}

\begin{document}

\allowdisplaybreaks

\newcommand{\arXivNumber}{1707.02893}

\renewcommand{\thefootnote}{}

\renewcommand{\PaperNumber}{083}

\FirstPageHeading

\ShortArticleName{Twists of Elliptic Curves}

\ArticleName{Twists of Elliptic Curves\footnote{This paper is a~contribution to the Special Issue on Modular Forms and String Theory in honor of Noriko Yui. The full collection is available at \href{http://www.emis.de/journals/SIGMA/modular-forms.html}{http://www.emis.de/journals/SIGMA/modular-forms.html}}}

\Author{Max KRONBERG~$^\dag$, Muhammad Afzal SOOMRO~$^\ddag$ and Jaap TOP~$^\dag$}

\AuthorNameForHeading{M.~Kronberg, M.A.~Soomro and J.~Top}

\Address{$^\dag$~Johan Bernoulli Institute for Mathematics and Computer Science,\\
\hphantom{$^\dag$}~Nijenborgh~9, 9747 AG Groningen, The Netherlands}
\EmailD{\href{mailto:m.c.kronberg@rug.nl}{m.c.kronberg@rug.nl}, \href{mailto:j.top@rug.nl}{j.top@rug.nl}}

\Address{$^\ddag$~Quaid-e-Awam University of Engineering, Science \& Technology (QUEST), \\
\hphantom{$^\ddag$}~Sakrand Road, Nawabshah, Sindh, Pakistan}
\EmailD{\href{mailto:m.a.soomro@quest.edu.pk}{m.a.soomro@quest.edu.pk}}

\ArticleDates{Received July 10, 2017, in f\/inal form October 23, 2017; Published online October 25, 2017}

\Abstract{In this note we extend the theory of twists of elliptic curves as presented in various standard texts for characteristic not equal to two or three to the remaining characteristics. For this, we make explicit use of the correspondence between the twists and the Galois cohomology set $H^1\big(\operatorname{G}_{\overline{K}/K}, \operatorname{Aut}_{\overline{K}}(E)\big)$. The results are illustrated by examples.}

\Keywords{elliptic curve; twist; automorphisms; Galois cohomology}

\Classification{11G05; 11G25; 14G17}

\begin{flushright}
\begin{minipage}{80mm}
\it Dedicated to Noriko Yui. The third author of this note was a postdoc with her at Queen's University during 1989--1990.
\end{minipage}
\end{flushright}

\renewcommand{\thefootnote}{\arabic{footnote}}
\setcounter{footnote}{0}

\section{Introduction}
Throughout this paper $K$ will be a perfect f\/ield and we always f\/ix a separable closure of $K$, which we denote by $\overbar{K}$. For the absolute Galois group of $\overbar{K}$ over $K$ we write $\absGal$. Let $E/K$ be an elliptic curve over $K$. A twist of $E$ is an elliptic curve $\Etw/K$ that is isomorphic to $E$ over $\overbar{K}$. In other words, it is an elliptic curve over $K$ with $j$-invariant~$j(E)$. Two such twists are considered equal if they are isomorphic over $K$. We denote the set of twists by $\Twist(E/K)$. For the automorphism group of $E$ we write $\Aut_{\overbar{K}}(E)$. The elements of $\Twist(E/K)$ are in one-to-one correspondence with the classes in $H^1\big(\absGal, \Aut_{\overbar{K}}(E)\big)$ \cite[Chapter~X, Section~2]{silverman09}. We want to remark that our notation dif\/fers from the notation used by Silverman. He denotes the set of twists by $\Twist(E/K,O)$.

Recently there has been quite some interest in twists of not only elliptic curves, but also curves in general and even in twists of algebraic varieties over various f\/ields \cite{elisa3,GouveaYui, Rachel,Elise2, Elise,MeagherTop10,Carlo}. And besides twists of varieties, twists of maps appear to have an increasing role in arithmetic dynamics \cite{Levy,Manes,SilComp,Stout}, \cite[Section~4.9]{SilAD}.

The simplest nontrivial example of twists is provided by the case of elliptic curves, where an early account of it was given by J.W.S.~Cassels in \cite[Part~II, Section~9]{cassels66}. We brief\/ly recall some of this theory here; in the case that the automorphism group of the elliptic curve is cyclic this is covered in various standard textbooks on elliptic curves. Although it is certainly known to most experts how to extend this theory to the cases where the automorphism group is noncyclic (and hence not even abelian), there seems to be no adequate reference for this and we hope to f\/ill this gap. Note, by the way, that Ian Connell at McGill University (Montreal) in the late 1990's wrote extensive unpublished lecture notes on elliptic curves, including a~long chapter on twists~\cite{Connell}. Inspired by those notes John Cremona implemented various algorithms for dealing with twists of elliptic curves in sage; see the lines 557--604 of~\cite{Cremona} for sage code related to Section~\ref{section2} of the present paper, and the lines 507--554 of~\cite{Cremona} for sage code related to our Section~\ref{section3}.

For a positive integer $n$ coprime to the characteristic of $K$ we denote by $\mu_n(\overbar{K})$ the group of $n$-th roots of unity in $\overbar K^\times$ and by $\zeta_n$ a generator of this group. In \cite{silverman09}, Silverman only presents an explicit description of the twists of an elliptic curve $E/K$ in $\charact{K}\neq2,3$. The main reason is that this condition implies $\Aut_{\overbar{K}}(E)\cong\mu_n(\overbar{K})$, for some $n\in\{2,4,6\}$, even as $\absGal$-modules. In characteristic $2$ the group $\Aut_{\overbar{K}}(E)$ is either isomorphic to $\Z/2\Z$ or to a non-abelian group of order~$24$. In characteristic three the group $\Aut_{\overbar{K}}(E)$ is either equals $\mu_2(\overbar K)\cong\Z/2\Z$ or it is a~non-abelian group of order~$12$. By explicitly describing $H^1\big(\absGal, \Aut_{\overbar{K}}(E)\big)$ in these remaining cases, we complete the description presented in~\cite{silverman09}.

We start by considering twists of elliptic curves with $j$-invariant equal to zero in characteristic three and two. We then consider the twists corresponding to normal subgroups of $\Aut_{\overbar{K}}(E)$. The possible subgroups correspond to quadratic, cubic and sextic twists.

The main results of this note can be found in Propositions~\ref{2.1} and~\ref{3.1} which describe twists over f\/inite f\/ields of characteristic three respectively two, provided the automorphism group of the elliptic curve is non-abelian, and in Propositions~\ref{charthree} and~\ref{chartwo}, where we count the number of twists of any elliptic curve over a f\/inite f\/ield. Sections~\ref{section2} and~\ref{section3} also indicate how these twists can be given explicitly in various non-trivial cases. Next, Propositions~\ref{4.1},~\ref{5.1} and~\ref{6.1} answer the question under what conditions a~potentially quadratic or cubic twist of a given elliptic curve is in fact still isomorphic to the curve one starts with.

Parts of the results of this paper originate from the PhD thesis of the second author \cite[Section~2.6]{Soomro13}.

\section{Twists in characteristic three}\label{section2}
We start by considering elliptic curves over f\/inite f\/ields $\F_{3^n}$. This is done by analysing the following central example. By~\cite{MeagherTop10}, the twists of an elliptic curve $E$ over a f\/inite f\/ield $\F=\F_q$ of cardinality $q$ are in one-to-one correspondence with the Frobenius conjugacy classes in $\Aut_{\overbar{\F}}(E)$. By def\/inition a Frobenius conjugacy class is obtained by f\/ixing some $\tau\in\Aut_{\overbar{\F}}(E)$, then its Frobenius conjugacy class consists of all
\begin{gather*}
\big\{ \sigma^{-1}\tau \big({}^{\Frob}\sigma\big)\big\}.
\end{gather*}
Here $\Frob$ is the f\/ield automorphism of $\overbar{\F}$ raising any element to its $q$th power. It acts on an automorphism $\sigma$ of $E$ by acting on the coef\/f\/icients of the rational functions def\/ining $\sigma$. We will compute these Frobenius conjugacy classes for all possible actions of the absolute Galois group.

\begin{Proposition}\label{2.1} 	The elliptic curve
\begin{gather*}
E/\mathbb{F}_3\colon \ y^2=x^3-x
\end{gather*}
has $j(E)=0$, and it has precisely twelve automorphisms. These are given by
\begin{gather*}
\map{\Phi_{u,r}}{E}{E,}{(x,y)}{\big(u^2x+r,u^3y\big),}
\end{gather*}
where $u^4=1$ and $r\in\F_3$. Let $n\geq 1$ and $q=3^n$. For $u\in\F_9$ with $u^4=1$ and $r\in\F_3$, put{\samepage
\begin{gather*}
C_{u,r}=\big\{\Phi_{u',r'}^{-1}\circ\Phi_{u,r}\circ\Phi_{u'^q,r'^q}\,|\, (u')^4=1,\,r'\in\F_3 \big\}
\end{gather*}
$($this is the Frobenius conjugacy class of $\Phi_{u,r}$ over $\F_q)$.}

If $n$ is odd, then the Frobenius conjugacy classes of $E/\F_q$ are
\begin{gather*}
	C_{1,0}=\{\Phi_{1,0}, \Phi_{-1,0}\},\qquad C_{1,1}=\{\Phi_{1,1},\Phi_{-1,-1}\},\qquad C_{1,-1}=\{\Phi_{1,-1},\Phi_{-1,1}\},\\
 C_{i,0}=\{\Phi_{u,r}\,|\, u^2=-1,\,r \in \F_3\}.
\end{gather*}
In particular there are precisely three non-trivial twists of $E$ over $\F_q$ in case $n$ is odd, which are given as follows.

Consider the cocycle defined by $\Frob\mapsto \Phi_{u,0}$. The corresponding twist is given by
\begin{gather*}\Etw\colon \ y^2=x^3+x
\end{gather*}
and the corresponding isomorphism is defined over a quadratic extension.

Analogously the cocycle $\Frob\mapsto\Phi_{1,1}$ corresponds to the twist
\begin{gather*}
	\Etw\colon \ y^2=x^3-x-1
\end{gather*}
and the cocycle $\Frob\mapsto\Phi_{1,-1}$ corresponds to the twist
\begin{gather*}
	\Etw\colon \ y^2=x^3-x+1,
\end{gather*}
where both isomorphisms are defined over a cubic extension.

In case $n$ is even, fix $i\in\F_q$ with $i^2=-1$. The Frobenius conjugacy classes of $E/\F_q$ are
\begin{gather*}
C_{1,0}=\{\Phi_{1,0}\},\qquad C_{-1,0}=\{ \Phi_{-1,0}\},\qquad C_{1,1}=\{\Phi_{1,1},\Phi_{1,-1}\},\\ C_{-1,1}=\{\Phi_{-1,1},\Phi_{-1,-1}\},\qquad
C_{i,0}=\{\Phi_{i,r}\,|\, r \in \F_3\},\qquad C_{-i,0}=\{\Phi_{-i,r}\,|\, r \in \F_3\}.
\end{gather*}

The corresponding isomorphisms are defined over a field extension of degree $1$, $2$, $3$, $6$, $4$, $4$, respectively.

Equivalently, since we are considering here the number of $\F_{3^n}$-isomorphism classes of elliptic curves with $j$-invariant~$0$, i.e., of supersingular elliptic curves over $\F_{3^n}$, this shows there are $4$ supersingular curves when~$n$ is odd and~$6$ such curves when~$n$ is even. This is of course well known; it is consistent with the tables presented in~{\rm \cite{Schoof}}.
\end{Proposition}
\begin{proof} The statements about the $j$-invariant and about the number of automorphisms are easy; compare, e.g., \cite[Appendix~A, Proposition~1.2]{silverman09}.

Since $\Phi_{u,r}^{-1}=\Phi_{u^{-1},-u^2r}$ and $ ^{\Frob}\Phi_{u,r}=\Phi_{u^3,r^3}=\Phi_{u^{-1},r}$, one can directly calculate the Frobenius conjugacy class of $\Phi_{u,r}$, depending on $q$ being an even or an odd power of $3$.

To verify that indeed the curves presented in the statement of the proposition correspond to the given Frobenius conjugacy classes, one needs to use an isomorphism $\psi\colon E\to \Etw$ and check that $\big( ^{\Frob}\psi\big)^{-1}\circ\psi$ is in the Frobenius conjugacy class. For example, with $\Etw\colon y^2= x^3+x$ one can use $\psi\colon E\to\Etw$ is given by $(x,y)\mapsto(ix,-iy)$ (with $i^2=-1$). A direct computation shows that
$\big( ^{\Frob}\psi\big)^{-1}\circ\psi=\Phi_{i,0}$. The other cases are done similarly.
\end{proof}

Note that we only presented explicit equations for the twists in the case of an extension of~$\F_3$ of odd degree. If the degree is even, such an equation will in general (as expected) depend on the f\/ield~$\F_q$.

Since we are considering here the number of $\F_{3^n}$-isomorphism classes of elliptic curves with $j$-invariant $0$, i.e., of supersingular elliptic curves over $\F_{3^n}$, the proposition shows there are $4$ supersingular curves when $n$ is odd and $6$ such curves when $n$ is even. This is of course well known; it is consistent with the tables presented in~\cite{Schoof}.

We now more generally consider the case that $E/K$ is an elliptic curve def\/ined over a f\/ield~$K$ with $\charact(K)=3$ such that $\#\Aut_{\overbar{K}}(E)=12$. This means (compare \cite[Appendix~A]{silverman09}) that $j(E)=0$ and $E$ is given by an equation $y^2=x^3+ax+b$, where $a,b\in K$. Thus, there exists an isomorphism $\psi\colon E\to E'$, where $E'\colon y^2=x^3-x$. We are interested in the possibilities for the f\/ield extension where the isomorphism is def\/ined. By \cite[Appendix~A, Proposition~1.2]{silverman09}, we have $\psi(x,y)=(u^2x+r,u^3y)$, where $u^4=-\frac1a$ and $r^3+ar+b=0$. Thus, we see that the degree of the f\/ield extension depends on the existence of a $K$-rational $2$-torsion point on~$E$.

In the case that $E[2](K)$ is trivial, we have $b\neq 0$. As $\psi^{-1}(x,y)=(v^2x+w,v^3y)$ where $v=u^{-1}$ and $w^3-w-v^6b=0$, both $\psi$ and $\psi^{-1}$ are def\/ined over the Artin--Schreier extension of $K(u)$ def\/ined by $w^3-w-v^6b=0$.

In the case that $E[2](K)$ is non-trivial, we may assume $b=0$ and any such isomorphism is def\/ined over~$K(u)$.

The f\/ield $K(u)$ depends in both cases only on $a$ and is a degree four extension if~$a$ is not a square in~$K$.

To complete the picture in characteristic three, note that any elliptic curve $E/K$ in charac\-te\-ris\-tic~$3$ with $j(E)\neq 0$ satisf\/ies $\Aut_{\overbar{K}}(E)=\pm 1$ and therefore $H^1(\absGal,\Aut_{\overbar{K}}(E))\cong K^\times/{K^\times}^2$. In particular, summarizing most of the discussion above for the special case of a f\/inite f\/ield, one obtains the following.
\begin{Proposition}\label{charthree} Let $q=3^n$ and suppose $E/\F_q$ is an elliptic curve. Then
\begin{gather*}
\#\Twist(E/\F_q)=
\begin{cases}
2 & \text{if} \ j(E)\neq 0,\\
4 & \text{if} \ j(E)=0 \ \text{and} \ n \ \text{is odd},\\
6 & \text{if} \ j(E)=0 \ \text{and} \ n \ \text{is even}.
\end{cases}
\end{gather*}
\end{Proposition}
\section{Twists in characteristic two} \label{section3}
In order to describe twists in characteristic two, we start by considering the central example of a~supersingular elliptic curve over the f\/ield with two elements. As in the case of characteristic three, this is done by computing the Frobenius conjugacy classes in all possible cases for the action of $\absGal$ on $\Aut_{\overbar{K}}(E)$. After this description we turn to isomorphisms between an arbitrary elliptic curve over a f\/ield with characteristic two and this particular example. Just is was done for characteristic three, the example is formulated as a proposition, as follows.

\begin{Proposition}\label{3.1} The elliptic curve
\begin{gather*}
E/\mathbb{F}_2\colon \ y^2+y=x^3
\end{gather*}
has $j(E)=0$ and it has exactly $24$ automorphisms. These are described as
\begin{gather*}
\map{\Phi_{u,r,t}}{E}{E,}{(x,y)}{\big(u^2x+r,y+u^2r^2x+t\big),}
\end{gather*}
where $u\in \mathbb{F}^*_4$, $r\in \mathbb{F}_4$ and $t^2+t+r^3=0$. The action of the Galois group $\absGal[\F_q]$ on $\Aut_{\overbar{\mathbb{F}_q}}(E)$ is trivial in case $n$ is even, and nontrivial if $n$ is odd.

In case $n$ is odd, there are exactly three Frobenius conjugacy classes $C_{u,r,t}$ in $\Aut_{\overbar{\mathbb{F}}_q}(E)$, namely the conjugacy class $C_{1,0,0}$ containing the identity, the class $C_{1,\omega,\omega}$ containing $\Phi_{1,\omega,\omega}$, and the class $C_{1,\omega,\omega^2}$
containing $\Phi_{1,\omega,\omega^2}$. Here $\omega\in\F_4$ is a primitive $3$rd root of unity. The two Frobenius conjugacy classes corresponding to
non-trivial twists of $E$ over $\F_q$ yield twists of $E$ over $\F_q$ which are isomorphic to $E$ over a degree eight extension of $\F_q$.

In case $n$ is even, the Frobenius conjugacy classes coincide with the usual conjugacy classes in $\Aut_{\overbar{\F_q}}(E)$, which are $($with $C_{u,r,t}$ denoting the conjugacy class containing $\Phi_{u,r,t}$ and $\omega\in\F_4$ a~chosen primitive $3$rd root of unity$)$
\begin{gather*}
C_{1,0,0},\
 C_{1,0,1},\
 C_{\omega^2,0,1},\
 C_{\omega,0,1},\
 C_{\omega,0,0},\
 C_{\omega^2,0,0},\
 C_{1,1,\omega}.
\end{gather*}
The twists of $E/\F_q$ corresponding to these conjugacy classes are isomorphic to $E$ over an extension of $\F_q$ of degree $1$, $2$, $6$, $6$, $3$, $3$, $4$, respectively.

In particular $E/\F_{2^n}$ has two non-trivial twists if $n$ is odd, and six non-trivial twists in case~$n$ is even.
\end{Proposition}

\begin{proof}The statements about $j(E)$ and $\Aut(E)$ are immediate; compare \cite[Appendix~A]{silverman09}.

In any automorphism $\Phi_{u,r,t}$ one has $r\in \mathbb{F}_4$ hence $r^3=1$ if $r\neq0$ and $r^3=0$ for $r=0$. So the equality $t^2+t+r^3$ shows $t\in\F_2$ for $r=0$ and $t\in\F_4$ for $r\neq 0$. Therefore all automorphisms are def\/ined over~$\F_4$.

To obtain the Frobenius conjugacy classes we f\/irst assume $n$ is odd, and we write $C_{u,r,t}$ for the Frobenius conjugacy class containing
$\Phi_{u,r,t}$. Let $\omega\in\F_4$ be a f\/ixed primitive $3$rd root of unity. A direct calculation shows
\begin{gather*}
C_{1,0,0}=\big\{ \Phi_{u,r,t}^{-1}\Phi_{1,0,0}\Phi_{u^2,r^2,t^2}\,| \,u\in\F_4^*,\,r\in\F_4,\, t^2+t+r^3=0\big\} \\
\hphantom{C_{1,0,0}}{} = \big\{\Phi_{u^2,ur^2+r,ur}\,|\,u\in\F_4^*,\,r\in\F_4\big\},
\end{gather*}
so $C_{1,0,0}$ consists of
\begin{gather*}
\big\{\Phi_{1,0,0}, \Phi_{1,0,1}, \Phi_{1,1,\omega}, \Phi_{1,1,\omega^2},
\Phi_{\omega^2,0,0}, \Phi_{\omega^2,\omega^2,\omega}, \\
\hphantom{\big\{}{}
 \Phi_{\omega^2,\omega^2,\omega^2}, \Phi_{\omega^2,0,1},
\Phi_{\omega,0,0}, \Phi_{\omega,\omega,\omega^2}, \Phi_{\omega,0,1}, \Phi_{\omega,\omega,\omega} \big\}.
\end{gather*}

The other two Frobenius conjugacy classes are given by
\begin{gather*}
C_{1,\omega,\omega}=\big\{\Phi_{1,\omega,\omega},\Phi_{\omega,\omega^2,\omega},\Phi_{\omega,1,\omega^2},
\Phi_{\omega,\omega,\omega},\Phi_{\omega^2,1,\omega},\Phi_{1,\omega^2,\omega^2}\big\},\\
 C_{1,\omega,\omega^2}=\big\{\Phi_{1,\omega,\omega^2},\Phi_{\omega,\omega^2,\omega^2},\Phi_{\omega^2,1,\omega^2},
 \Phi_{\omega^2,\omega,\omega},\Phi_{\omega,1,\omega},\Phi_{1,\omega,\omega^2}\big\}.
\end{gather*}

To see that a twist of $E/\F_q$ corresponding to one of the latter two Frobenius conjugacy classes is indeed isomorphic to $E$ over $\F_{q^8}$ and
not over a smaller extension of $\F_q$, consider a cocycle def\/ined by $\Frob\mapsto \Phi$ (with $\Phi$ in one of the given classes $C_{u,r,s}$). Using the cocycle condition one f\/inds that the cocycle sends $\Frob^j$ (for $j\geq 1$) to ${}^{(\id+\Frob+\dots+\Frob^{j-1})}\Phi$. This is a nontrivial automorphism for $j\leq 7$ and the trivial one for $j=8$. The assertion about the twists follows.

We now consider the case $n$ is even. The conjugacy classes in $\Aut(E)$ are
\begin{gather*}
	C_{1,0,0}=\{\Phi_{1,0,0}\},\\
 C_{1,0,1}=\{\Phi_{1,0,1}\},\\
 C_{\omega^2,0,1}=\big\{\Phi_{\omega^2,0,1},\Phi_{\omega^2,1,\omega},\Phi_{\omega^2,\omega^2,\omega},\Phi_{\omega^2,\omega,\omega}\big\},\\
 C_{\omega,0,1}=\big\{\Phi_{\omega,0,1},\Phi_{\omega,1,\omega^2},\Phi_{\omega,\omega^2,\omega^2},\Phi_{\omega,\omega,\omega^2}\big\},\\
 C_{\omega,0,0}=\big\{\Phi_{\omega,0,0},\Phi_{\omega,1,\omega},\Phi_{\omega,\omega^2,\omega},\Phi_{\omega,\omega,\omega}\big\},\\
 C_{\omega^2,0,0}=\big\{\Phi_{\omega^2,0,0},\Phi_{\omega^2,1,\omega^2},\Phi_{\omega^2,\omega^2,\omega^2},\Phi_{\omega^2,\omega,\omega^2}\big\},\\
 C_{1,1,\omega}=\big\{\Phi_{1,1,\omega},\Phi_{1,1,\omega^2},\Phi_{1,\omega,\omega},
 \Phi_{1,\omega,\omega^2},\Phi_{1,\omega^2,\omega},\Phi_{1,\omega^2,\omega^2}\big\}.
\end{gather*}
So indeed there are exactly $6$ non-trivial twists of $E/\F_q$ in this case. The statement about the extension of $\F_q$ over which they
will be isomorphic to $E$ is shown exactly as the analogous statement for the odd $n$ case (in fact in the present situation it simply refers to the order of the elements in a certain conjugacy class).
\end{proof}

We present some comments regarding the argument above. First take $n=1$, so $q=2^n=2$. Since $-1=\Phi_{1,0,1}$ is in the same Frobenius conjugacy class as the identity, the elliptic curve~$E/\mathbb{F}_2$ has no non-trivial quadratic twist. Let us now consider the cubic twists of~$E$. Again we can see that the automorphisms of order $3$ are in the same conjugacy class of the identity and thus, $E/\F_2$~has no non-trivial cubic twists (and as the proposition states, any non-trivial twist of~$E/\F_2$ will only be isomorphic to~$E$ over extensions of~$\F_2$ of degree a multiple of~$8$).

Explicit equations for the two non-trivial twists of $E/\F_q$, in the case $q=2^n$ with $n$ odd, are as follows. Let
\begin{gather*} E_1\colon \ y^2+y=x^3+x
\end{gather*}
and
\begin{gather*}E_2\colon \ y^2+y=x^3+x+1.\end{gather*}
These elliptic curves indeed satisfy $j(E_1)=j(E_2)=j(E)=0$. Moreover counting points over~$\F_{2^n}$ (compare \cite[Section~V.2]{silverman09}) one f\/inds
$\#E(\F_{2^n})$ and $\#E_1(\F_{2^n})$ and $\#E_2(\F_{2^n})$ are three distinct numbers whenever $n$ is odd. So indeed the curves $E_j$ are distinct and nontrivial twists of~$E/\F_q$. Of course the same conclusion can be obtained by exhibiting an explicit isomorphism $\psi\colon E\to E_j$ and then determining the Frobenius conjugacy class containing ${}^{\Frob}\psi^{-1}\circ\psi$.

We will not try to present equations for all non-trivial twists of $E/\F_q$ in the case $q=2^n$ with~$n$ even. They depend on $q$. Instead we treat one example.

Let $q=4$. In this case ${\rm Fr}\mapsto -1=\Phi_{1,0,1}$ def\/ines a non-trivial cocycle class. The corresponding twist is given by
\begin{gather*}
\Etw\colon \ y^2+y=x^3+\omega
\end{gather*}
(with $\omega$ a primitive $3$rd root of unity), since
\begin{gather*}
\psi\colon \ (x,y)\mapsto (x,y+\tau),
\end{gather*}
with $\tau \in \mathbb{F}_4$ satisfying $\tau^2+\tau+\omega=0$, def\/ines an isomorphism $\psi\colon E\rightarrow \Etw$, and
\begin{gather*}
\big({}^{\rm Fr}\psi\big)^{-1}\circ \psi=-1.
\end{gather*}

Thus, $E/\F_4$ has a non-trivial quadratic twist.

Let $K$ be an f\/ield with $\charact(K)=2$ and consider the elliptic curves $E\colon y^2+ay=x^3+bx+c$ with $a\neq0$ and $E'\colon y^2+y=x^3$. Since $j(E)=0=j(E')$ these elliptic curves are isomorphic and, by Silverman \cite[Appendix A, Proposition~1.2]{silverman09}, for an isomorphism $\psi\colon E\to E'$ we have $\psi(x,y)=(u^2x+s^2,ay+u^2s+t)$, where $u^3=a$, $s^4+as+b=0$ and $t^2+at+s^6+bs^2+c=0$. Moreover from the information presented in Proposition~\ref{3.1} it follows that in case~$K$ is a f\/inite f\/ield, such $u$, $s$, $t$ exist in an extension of degree at most~$8$ resp.~$6$, depending on the action of the Galois group on the automorphism group of~$E$.

Again, we summarize the main results given here for the case of a f\/inite f\/ield, as follows. Here as before a crucial remark is that for~$E/K$ an elliptic curve in characteristic~$2$, the automorphism group over the separable closure is $\pm 1$ unless $j(E)=0$.

\begin{Proposition}\label{chartwo}
Let $q=2^n$ and suppose $E/\F_q$ is an elliptic curve. Then
\begin{gather*}
\#\Twist(E/\F_q)=
\begin{cases}
2 & \text{if} \ j(E)\neq 0,\\
3 & \text{if} \ j(E)=0 \ \text{and} \ n \ \text{is odd},\\
7 & \text{if} \ j(E)=0 \ \text{and} \ n \ \text{is even}.
\end{cases}
\end{gather*}
\end{Proposition}

\section{Capitulation of quadratic twists}\label{section4}
Let $K$ be a f\/ield and $\overbar{K}$ a~separable closure of $K$. Given an elliptic curve $E/K$, the inclusion $\lmr\subset \Aut_{\overbar{K}}(E)$ induces a map
\begin{gather*}
H^1\big(\absGal, \lmr\big)\longrightarrow H^1\big(\absGal, \Aut_{\overbar{K}}(E)\big).
\end{gather*}
The set of quadratic twists of $E$, i.e.,
\begin{gather*}
QT(E)=\big\{\Etw/K \, | \, \exists\, L/K \text{ with } [L:K]=2 \text{ such that } \Etw\cong _{L}E \big\}/_{\cong_K}
\end{gather*}
is a subset of $\Twist(E/K)$; this subset of quadratic twists corresponds to the image of \linebreak $H^1\big(\absGal, \lmr\big)$ in $H^1\big(\absGal, \Aut_{\overbar{K}}(E)\big)$ under the map just given. Here we consider the question whether $\Etw\in QT(E)$ can be isomorphic to $E$ over the ground f\/ield. In other words, when does $\Etw$ correspond to the trivial element in $H^1\big(\absGal, \Aut_{\overbar{K}}(E)\big)$, under the assumption that it comes from a non-trivial element in the group $H^1\big(\absGal, \lmr\big)=\operatorname{Hom}\big(\absGal, \Z/2\Z\big)$.

This question is analogous to a similar question in algebraic number theory and in function f\/ield arithmetic, namely the so-called capitulation (or principalization) problem for ideals, see, e.g., \cite{Bond,Bosca,LiHu,SW}.

Here is a result concerning capitulation of quadratic twists.
\begin{Proposition}\label{4.1} Let $E/K$ be an elliptic curve such that $\Aut_{\overbar{K}}(E)$ is abelian. Then the map
\begin{gather*}
i\colon \ H^1\big(\absGal, \lmr\big) \longrightarrow H^1\big(\absGal, \Aut_{\overbar{K}}(E)\big)
\end{gather*}
is injective except in the case when $\charact(K) \not \in \{2, 3 \}$, $j(E)=12^3$ and $\absGal$ acts non-trivially on~$\Aut_{\overbar{K}}(E)$.
\end{Proposition}

\begin{proof} We have the following long exact sequence of groups:
\begin{gather*}
\xymatrix@C=0.95pc{
&1 \ar[r] & H^0\big(\absGal, \lmr\big) \ar[r] & H^0\big(\absGal, \Aut_{\overbar{K}}(E)\big) &\\
& & H^0\big(\absGal,\faktor{\Aut_{\overbar{K}}(E)}{\lmr}\big) \ar@{<-} `l [lu] `[r] `[rru]-(7,0)^{\pi} `[ur] \ar[r] & H^1\big(\absGal, \lmr\big) &\\
& & H^1\big(\absGal, \Aut_{\overbar{K}}(E)\big) \ar@{<-} `l [lu] `[r] `[rru]-(8,0)^i `[ur] \ar[r] & H^1\big(\absGal,\faktor{\Aut_{\overbar{K}}(E)}{\lmr}\big). & &
}
\end{gather*}

Note that $H^0\big(\absGal, \lmr\big)=\Z/2\Z$. By \cite[Chapter~III, Corollary~10.2]{silverman09}, we have the following automorphism groups of an elliptic curve:
\begin{enumerate}\itemsep=0pt
\item[1)] $\Aut_{\overbar{K}}(E)=\lmr \cong\Z/2\Z$, when $j(E)\neq0,12^3$;
\item[2)] $\Aut_{\overbar{K}}(E)\cong\mu_4(\overbar K)$, when $j(E)=12^3$ and $\charact(K) \not \in \{2,3\}$;
\item[3)] $\Aut_{\overbar{K}}(E)\cong\mu_6(\overbar K)$, when $j(E)=0$ and $\charact(K)\not \in \{2,3\}$.
\end{enumerate}
We consider each case separately.

1.~Since $\#\Aut_{\overbar{K}}(E)=2$, the Galois group $\absGal$ acts trivially on $\Aut_{\overbar{K}}(E)$. Therefore, $\#H^0\big(\absGal,\faktor{\Aut_{\overbar{K}}(E)}{\langle -1 \rangle}\big)=1$. Hence, the map $i$ is injective. This proves the proposition in this case.

2.~First, suppose $\absGal$ acts trivially on $\Aut_{\overbar{K}}(E)\cong \mu_4=\big\{1,\zeta_4,\zeta_4^2,\zeta_4^3\big\}$. Again, here we see $\#H^0\big(\absGal,\faktor{\Aut_{\overbar{K}}(E)}{\langle -1 \rangle}\big)=2$. Hence, the f\/irst four groups in the long exact sequence presented in the f\/irst lines of this proof have order as indicated in the following diagram
\begin{gather*}
\xymatrix{1 \ar[r] & 2 \ar[r] & 4 \ar[r]^\pi & 2 \ar[r] &.
	}
\end{gather*}

This implies that $\pi$ is surjective; therefore, $i$ is injective. The proposition follows in this case.

Now, suppose $\absGal$ acts non-trivially on $\Aut_{\overbar{K}}(E)$. Thus, there exists an automorphism $\sigma \in \absGal$ such that
\begin{gather*}
\sigma(1)=1,\qquad \sigma(\zeta_4)=\zeta_4^3.
\end{gather*}
Therefore, $\# H^0\big(\absGal,\Aut_{\overbar{K}}(E)\big)=2$. Now, the action of $\sigma$ on
\begin{gather*}
\faktor{\Aut_{\overbar{K}}(E)}{\langle -1 \rangle}\cong\big\{ \big\{1,\zeta_4^2\big\},\big\{\zeta_4,\zeta_4^3\big\} \big\}
\end{gather*}
is
\begin{gather*}
\sigma\big(\big\{1,\zeta_4^2\big\}\big)=\big\{1,\zeta_4^2\big\},\qquad \sigma\big(\big\{\zeta_4,\zeta_4^3\big\}\big)=\big\{\zeta_4,\zeta_4^3\big\}.
\end{gather*}
We conclude that $\#H^0\big(\absGal,\faktor{\Aut_{\overbar{K}}(E)}{\langle -1 \rangle}\big)=2$. Hence, the f\/irst four groups in the long exact sequence presented at the beginning of this proof have order as indicated in the following diagram
\begin{gather*}
\xymatrix{1 \ar[r] & 2 \ar[r] & 2 \ar[r]^\pi & 2 \ar[r] &.
}
\end{gather*}
Thus, $\pi$ is the constant map; therefore, $\#{\Ker}(i)=2$ and $i$ is not injective. The proposition follows in this case.

3. First, if $\absGal$ acts trivially on $\Aut_{\overbar{K}}(E)\cong \mu_6=\langle\zeta_6\rangle$, then we have $\#H^0\big(\!\absGal, \Aut_{\overbar{K}}(E)\big)$ $=6$ and $\#H^0\big(\absGal,\faktor{\Aut_{\overbar{K}}(E)}{\langle -1 \rangle}\big)=3$. The f\/irst four groups in the long exact sequence from the start of this proof therefore have order
\begin{gather*}
\xymatrix{1 \ar[r] & 2 \ar[r] & 6 \ar[r]^\pi & 3 \ar[r] &.
}
\end{gather*}
This implies that $\pi$ is surjective; therefore, because the sequence is exact, $i$ is injective.

Now, suppose $\absGal$ acts non-trivially on $\Aut_{\overbar{K}}(E)$. Let $\sigma \in \absGal$ acts non-trivially on $\Aut_{\overbar{K}}(E)$. Then we have
\begin{gather*}
\sigma(1)=1,\qquad \sigma(\zeta_6)=\zeta_6^5.
\end{gather*}
Thus we get $\# H^0\big(\absGal,\Aut_{\overbar{K}}(E)\big)=2$. The action of $\sigma$ on
\begin{gather*}
\faktor{\Aut_{\overbar{K}}(E)}{\langle -1 \rangle}\cong\big\{ \big\{1,\zeta_6^3\big\}, \big\{\zeta_6,\zeta_6^4\big\}, \big\{\zeta_6^2,\zeta_6^5\big\}\big\}
\end{gather*}
is given by
\begin{gather*}
\sigma\big(\big\{1,\zeta_6^3\big\}\big)=\big\{1,\zeta_6^3\big\},\qquad
\sigma\big(\big\{\zeta_6,\zeta_6^4\big\}\big)=\big\{\zeta_6^2,\zeta_6^5\big\},\qquad
\sigma\big(\big\{\zeta_6^2,\zeta_6^5\big\}\big)=\big\{\zeta_6,\zeta_6^4\big\},
\end{gather*}
implying that $\#H^0\big(\absGal,\faktor{\Aut_{\overbar{K}}(E)}{\langle -1 \rangle}\big)=1$. The f\/irst four groups in the long exact sequence used throughout this argument have orders
\begin{gather*}
\xymatrix{1 \ar[r] & 2 \ar[r] & 2 \ar[r]^\pi & 1 \ar[r] &.}
\end{gather*}
\looseness=1 We conclude that $\pi$ is surjective; hence, $i$ is injective. This completes the proof of the proposi\-tion.
\end{proof}

The condition in Proposition~\ref{4.1} that the automorphism group of $E$ should be abelian, means that one excludes only the cases $j(E)=0$ in ${\charact}(K) \in \left\{2,3\right\}$. We brief\/ly consider these two excluded cases here.

Suppose $\charact(K)=2$ and take $E/K\colon y^2+y=x^3$. As in Section~\ref{section3} one f\/inds that $\absGal$ acts trivially on $\Aut_{\overbar{K}}(E)$ precisely when $\F_4\subset K$. In that case no capitulation of quadratic twists occurs. However, when Galois acts non-trivially on this automorphism group, then as in the case of a f\/inite f\/ield studied in Section~\ref{section3} where we saw that $-1$ and $1$ are in the same Frobenius conjugacy class, capitulation occurs.

Similarly, suppose $\charact(K)=3$ and take $E\colon y^2=x^3-x$. Again, there is capitulation of quadratic twists precisely when the Galois action on the automorphism group is non-trivial, which happens precisely when~$\F_9\not\subset K$.

\begin{Example}\label{example1} Take
\begin{gather*}
E/\Q\colon \ y^2=x^3-x.
\end{gather*}
Then
\begin{gather*}
\Aut_{\overbar{\Q}}(E)= \{\pm 1, \pm \iota \},
\end{gather*}
where $\iota \colon E \rightarrow E$ is def\/ined by $(x,y)\mapsto (-x,iy)$ for a f\/ixed choice of a primitive $4$th root of unity $i \in \overbar{\Q}$.

For $d\in \Q^*$ write
\begin{gather*}
E^{(d)}\colon \ y^2=x^3-d^2x.
\end{gather*}
Then $E^{(d)}$ is a twist of $E/\Q$, since $\psi_d\colon E\rightarrow E^{(d)}$ def\/ined as
\begin{gather*}
\psi_d(x,y)=\big(dx,d\sqrt{d}y\big)
\end{gather*}
is an isomorphism between $E$ and $E^{(d)}$.

If $\sigma \in \absGal[\Q]$, then
\begin{gather*}
\big(^\sigma\psi_d\big)^{-1} \circ \psi_d= \begin{cases}
\hphantom{-}1 & \text{if} \ \sigma\big(\sqrt{d}\big)=\sqrt{d}, \\
-1 & \text{if} \ \sigma\big(\sqrt{d}\big)=-\sqrt{d}.
\end{cases}
\end{gather*}
So $E^{(d)}$ corresponds to the cocycle class of
\begin{gather*}
\sigma \mapsto \frac{\sigma\big(\sqrt{d}\big)}{\sqrt{d}} \in \Aut_{\bar{\Q}}(E).
\end{gather*}

In the case $d=-1$, this cocycle is a coboundary, since
\begin{gather*}
\frac{\sigma\big(\sqrt{-1}\big)}{\sqrt{-1}}=(^\sigma\iota)^{-1}\circ \iota.
\end{gather*}
So $E^{(-1)} \cong E$ over $\Q$, which is, of course, evident from the equation.
\end{Example}

\begin{Example}\label{example2}Take $q$ a power of an odd prime, and
\begin{gather*}
E/\Fq\colon \ y^2=x^3-x.
\end{gather*}
The Galois group $\absGal[\Fq]$ acts non-trivially on $\Aut_{\Fqbar}(E)$ if and only if $-1$ is not a square in $\Fq$. We have
\begin{gather*}
 \sqrt{-1} \not \in \Fq \iff q\equiv 3 \pmod 4.
\end{gather*}
 For $d\in \Fq^{*}$, def\/ine $E^{(d)}/\Fq$ as before. This provides a quadratic twist as in Example~\ref{example1}.

If $d$ is not a square and $q\equiv 1\pmod4$, then $E^{(d)}$ is the (unique) non-trivial quadratic twist of $E/\Fq.$

If $d$ is not a square and $q\equiv 3\pmod4$, then $-d$ is a square. Therefore, we have $E^{(d)}=E^{(-d)}\cong E$ over $\Fq$. So for $q\equiv 3\pmod4$, a non-trivial quadratic twist of $E/\Fq$ does not exist.
\end{Example}

\section{Capitulation of cubic twists}\label{section5}

Let $E/K$ be an elliptic curve such that $\Aut_{\overbar{K}}(E)$ contains a subgroup $C_3$ of order $3$. This implies $j(E)=0$ by \cite[Chapter~III, Corollary~10.2]{silverman09}. Thus, we restrict ourselves in this section to elliptic curves $E$ with $j(E)=0$. Note that in the case of $\charact(K)=2,3$ the group $\Aut_{\overbar{K}}(E)$ is not abelian; thus, the considered exact sequence is an exact sequence of pointed sets.
The non-abelian cohomology needed to describe twists in this situation, is, e.g., described in Serre's books \cite[Chapter~I, Section~5]{SerreGaloisCohom} and \cite[Chapter~XIII]{SerreLocalFields}.

\begin{Proposition}\label{5.1}
Let $E/K$ be an elliptic curve with $j(E)=0$. There is a unique $($and therefore normal$)$ subgroup $C_3\subset{\Aut_{\overbar{K}}(E)}$ of order $3$. The map
\begin{gather*}
	i\colon \ H^1\big(\absGal, C_3\big) \longrightarrow H^1\big(\absGal, \Aut_{\overbar{K}}(E)\big)
\end{gather*}
is injective except possibly when $\charact(K)=2$.
\end{Proposition}

\begin{proof} 1. First we consider the case $\charact(K)\neq2,3$. Then $\Aut_{\overbar{K}}(E)$ is cyclic of order $6$, so indeed $C_3$ as desired exists and is unique. If $\absGal$ acts trivially on $\Aut_{\overbar{K}}(E)$ we have
\begin{gather*}
	\#H^0\big(\absGal,\faktor{\Aut_{\overbar{K}}(E)}{\mu_3}\big)=2.
\end{gather*}

In the long exact sequence
\begin{gather*}
\xymatrix@C=0.95pc{
&1 \ar[r] & H^0\big(\absGal, C_3\big) \ar[r] & H^0\big(\absGal, \Aut_{\overbar{K}}(E)\big) &\\
& & H^0\big(\absGal,\faktor{\Aut_{\overbar{K}}(E)}{C_3}\big) \ar@{<-} `l [lu] `[r] `[rru]-(7,0)^{\pi} `[ur] \ar[r] & H^1\big(\absGal, C_3\big) &\\
& & H^1\big(\absGal, \Aut_{\overbar{K}}(E)\big) \ar@{<-} `l [lu] `[r] `[rru]-(8,0)^i `[ur] \ar[r] & H^1\big(\absGal,\faktor{\Aut_{\overbar{K}}(E)}{C_3}\big) & &}
\end{gather*}
the orders of the f\/irst few groups are
\begin{gather*}
\xymatrix{1 \ar[r] & 3 \ar[r] & 6 \ar[r]^\pi & 2 \ar[r] &}
\end{gather*}
 and thus, $\pi$ is surjective which implies $i$ is injective.

If on the other hand $\absGal$ acts non-trivially on $\Aut_{\overbar{K}}(E)$ then since Galois f\/ixes the $-1$-map, any $\sigma\in\absGal$ that acts non-trivially has to interchange the two non-trivial elements of $C_3$. This implies that $\absGal$ acts trivially on $\faktor{\Aut_{\overbar{K}}(E)}{\mu_3}$ and thus, in the long exact sequence presented earlier in this proof $\pi$ is surjective since $\# H^0\big(\absGal,\mu_3\big)=1$ and $\#H^0\big(\absGal,\Aut_{\overbar{K}}(E)\big)=1$. So again one concludes that $i$ is injective.

2. Let $\charact(K)=3$. This implies $\Aut_{\overbar{K}}(E)$ is a semi-direct product $C_3\rtimes C_4$ of cyclic groups of order $3$ and $4$ (see \cite[Appendix A, Example~A.1]{silverman09}). In particular it follows that the automorphism group has a unique subgroup $C_3$ of order $3$. If $\absGal$ acts trivially on $\Aut_{\overbar{K}}(E)$, then $i$ is injective. If $\absGal$ acts non-trivially on $\Aut_{\overbar{K}}(E)$ we will consider several cases.

First we consider the case that $\absGal$ acts trivially on $C_3$. This implies that any non-trivially acting $\sigma\in\absGal$ interchanges the two elements of order $4$ in $C_4$. This implies $\# H^0\big(\absGal,C_3\big)=3$ and $\# H^0\big(\absGal,\Aut_{\overbar{K}}(E)\big)=6$ and 	$\#H^0\big(\absGal,\faktor{\Aut_{\overbar{K}}(E)}{C_3}\big)=2$. This gives us the sequence of orders
\begin{gather*}
\xymatrix{1 \ar[r] & 3 \ar[r] & 6 \ar[r]^\pi & 2 \ar[r] &}
\end{gather*}
from which it follows that $\pi$ is surjective and thus $i$ is injective.

Now consider the case that all elements of $C_4$ are f\/ixed under the action of~$\absGal$. This implies that any $\sigma\in\absGal$ acting non-trivially on $\Aut_{\overbar{K}}(E)$ has to interchange the non-trivial elements of~$C_3$. Therefore, we get
\begin{gather*}
	\# H^0\big(\absGal,C_3\big)=1,\\
	\# H^0\big(\absGal,\Aut_{\overbar{K}}(E)\big)=4,\\
	\# H^0\big(\absGal,\faktor{\Aut_{\overbar{K}}(E)}{C_3}\big)=4.
\end{gather*}
Similar to the previous situation this implies that $i$ is injective.

In the case that neither $C_3$ nor $C_4$ are elementwise f\/ixed under the action of~$\absGal$, we easily get the following sequence of orders
\begin{gather*}
\xymatrix{1 \ar[r] & 1 \ar[r] & 2 \ar[r]^\pi & 1 \ar[r] &}
\end{gather*}
and thus again, $i$ is injective.
\end{proof}

So the result says that capitulation of cubic twists does not occur, except possibly in characteristic $2$. In fact Proposition~\ref{chartwo} implies that it also does not occur over f\/inite f\/ields in characteristic two.

\section{Capitulation of sextic twists}\label{section6}
Let $E/K$ be an elliptic curve such that $\Aut_{\overbar{K}}(E)$ has a normal subgroup of order $6$. This implies $j(E)=0$ and $\charact(K)\neq 2$.
A result analogous to Proposition~\ref{5.1} is the following.
\begin{Proposition}\label{6.1} Suppose $\charact(K)\neq 2$ and let $E/K$ be an elliptic curve with $j(E)=0$. There is a unique $($and therefore normal$)$ subgroup $C_6\subset{\Aut_{\overbar{K}}(E)}$ of order $6$. The map
\begin{gather*}
	i\colon \ H^1\big(\absGal, C_6\big) \longrightarrow H^1\big(\absGal, \Aut_{\overbar{K}}(E)\big)
\end{gather*}
is injective except in the two cases
\begin{enumerate}\itemsep=0pt
\item[$1)$] $\charact(K)=3$ and $\absGal$ acts trivially on $C_6$;
\item[$2)$] $\charact(K)=3$ and the only elements in $\Aut_{\overbar{K}}(E)$ fixed by $\absGal$ are $\pm 1$.
\end{enumerate}
\end{Proposition}

\begin{proof} In the case $\charact(K)\neq2,3$ we have for $E$ as above that $\Aut_{\overbar{K}}(E)$ is cyclic of order $6$. So the result is trivial in this case.

Now assume $\charact(K)=3$. In this case $\#\Aut_{\overbar{K}}(E)=12$ and this automorphism group contains a unique subgroup $C_6$ of order $6$. It is generated by the unique element of order $2$ and the subgroup $C_3$ of order $3$ in $\Aut_{\overbar{K}}(E)$. In the case that $\absGal$ acts trivially on $\Aut_{\overbar{K}}(E)$ we get once again that $i$ is injective. Thus, we now assume that the Galois action is non-trivial.

First case: the Galois action on $C_3$ is trivial. Then analogously to the case of cubic twists here $\#H^0(\absGal,C_6)=6$ and $\#H^0\big(\absGal,\Aut_{\overbar{K}}(E)\big)=6$. Obviously, $\absGal$ f\/ixes the residue classes modulo $C_6$. Thus the group $H^0\big(\absGal,\faktor{\Aut_{\overbar{K}}(E)}{C_6}\big)$ has order $2$. This gives us the sequence of orders
\begin{gather*}
\xymatrix{1 \ar[r] & 6 \ar[r] & 6 \ar[r]^\pi & 2 \ar[r] &},
\end{gather*}
implying in the same way as in earlier cases that $i$ is not injective, and the map
\begin{gather*}
\pi\colon \ H^0\big(\absGal, \Aut_{\overbar{K}}(E)\big)\to H^0\big(\absGal,\faktor{\Aut_{\overbar{K}}(E)}{H}\big)
\end{gather*} is constant.

Second case: the Galois action f\/ixes the points in a cyclic order $4$ subgroup of automorphisms. Then $\#H^0(\absGal,C_6)=2$ and $\#H^0(\absGal,\Aut_{\overbar{K}}(E))=4$. Furthermore the action on $\faktor{\Aut_{\overbar{K}}(E)}{C_6}$ is trivial, which gives us the sequence of orders
\begin{gather*}
\xymatrix{1 \ar[r] & 2 \ar[r] & 4 \ar[r]^\pi & 2 \ar[r] &}.
\end{gather*}
As before we conclude that $i$ is injective.

Third case: Neither $C_3$ nor an order $4$ subgroup $C_4$ are pointwise f\/ixed under the action of~$\absGal$. Then only $\pm1$ are f\/ixed in $\Aut_{\overbar{K}}(E)$ and $C_6$. Further, we see that the action on the quotient group again is trivial. This implies for the orders in the long exact sequence
\begin{gather*}
\xymatrix{1 \ar[r] & 2 \ar[r] & 2 \ar[r]^\pi & 2 \ar[r] &}.
\end{gather*}
So reasoning as before, $\pi$ is constant and $i$ is not injective. This case concludes the proof in characteristic $3$.
\end{proof}

In fact a slightly dif\/ferent proof of the same result may be obtained by observing that a sextic twist may be regarded as a cubic twist of a~quadratic one. We will not pursue this here.

\section{Other twists}\label{section7}

Although the techniques used in the previous sections require the (cyclic and Galois stable) subgroup used there to be normal, also in the non-normal cases one can draw conclusions.

We restrict ourselves to providing two examples.
\begin{Example} Take $q=3^n$ and consider $E/\F_q$ given by $y^2=x^3-x$. The automorphism $\Phi_{i,0}\colon (x,y)\mapsto (-x,\sqrt{-1}y)$ generates a Galois stable subgroup $H$ of $\Aut_{\overbar{\F_q}}(E)$ of order $4$. Then $\Frob\mapsto \Phi_{i,0}$ def\/ines a cocycle in $H^1(\absGal,H)$ and in $H^1(\absGal, \Aut_{\overbar{\F_q}}(E))$. In Proposition~\ref{2.1} we saw that this corresponds to a non-trivial twist.
\end{Example}

\begin{Example}Similarly we put $q=2^n$ and $E/\F_q$ given by $y^2+y=x^3$. With $\omega\in\overbar{\F_q}$ a~primitive third root of unity, the automorphism $\Phi_{\omega^2,0,1}\colon (x,y)\mapsto (\omega x,y+1)$ generates a Galois stable subgroup $H$ of $\Aut_{\overbar{\F_q}}(E)$
of order $6$, and $\Frob\mapsto \Phi_{\omega^2,0,1}$ def\/ines a cocycle in $H^1(\absGal,H)$ and in $H^1(\absGal, \Aut_{\overbar{\F_q}}(E))$.

Proposition~\ref{3.1} shows that for odd $n$ this results in a trivial twist, and for $n$ even one obtains a non-trivial twist.
\end{Example}

\subsection*{Acknowledgements}
It is a pleasure to thank Joe Silverman, John Cremona, Nurdag\"{u}l Anbar Meidl, and Jan Stef\/fen M\"{u}ller for helpful comments, advise, and support while preparing this paper. We are especially grateful for the many useful suggestions of the two referees of the f\/irst draft of this, which we hope have improved the readability of the text a lot.

\pdfbookmark[1]{References}{ref}
\LastPageEnding

\end{document}